\documentclass[a4paper, 12pt]{article}
\usepackage{amsmath}
\usepackage{amssymb}
\usepackage{amsthm}
\usepackage[top=20truemm, left=17truemm, right=17truemm]{geometry}
\usepackage{amscd}
\newtheorem{dfn}{Definition}[section]
\newtheorem{thm}[dfn]{Theorem}
\newtheorem{lem}[dfn]{Lemma}
\newtheorem{rem}[dfn]{Remark}
\newtheorem{cor}[dfn]{Corollary}

\newtheorem{prop}[dfn]{Proposition}
\newtheorem{ex}[dfn]{Example}
 
\usepackage{amscd}
\usepackage{graphicx}
\usepackage[all]{xy}
\usepackage{xcolor}

\title{Homotopy of inner automorphism groups of Cuntz algebras}

\date{}
\author{Taro Sogabe}
\begin{document}
\maketitle
\abstract{We show that the inner automorphism groups  and automorphism groups of the Cuntz algebras have the same homotopy groups.
In particular, the homotopy types of the inner automorphism groups and the projective unitary groups of Cuntz algebras are different.}
\section{Introduction}
The automorphism group $\operatorname{Aut}(A)$ of a C*-algebra $A$ is large in general, and several previous researches focus on its homotopy theory.
For a class of C*-algebras, called unital Kirchberg algebras, one can compute the homotopy group by using M. Dadarlat's formula:
\[\pi_i(\operatorname{Aut}(A))\cong KK^i(C_\nu A, SA),\quad (i\geq 1)\]
(see \cite{Dh}).
The left hand side is the i-th homotopy group of the topological group $\operatorname{Aut}(A)$ with the point wise norm topology,
and the right hand side is a certain computable abelian group, called KK-group, that helps us to study C*-algebras from algebraic point of view (see \cite{B} for the basics on K-theory and KK-theory).

Applying M. Dadarlat's formula to the Cuntz algebra $\mathcal{O}_n$ (see Sec. \ref{Pre} for the definition of this algebra),
one has
\begin{equation}
    \pi_{2k}(\operatorname{Aut}(\mathcal{O}_n))=0,\quad \pi_{2k+1}(\operatorname{Aut}(\mathcal{O}_n))=\mathbb{Z}/(n-1)\mathbb{Z}, \quad (k\geq 0)\label{dada}\end{equation}
which are different from the homotopy groups of projective unitary group $PU(\mathcal{O}_n):=U(\mathcal{O}_n)/\mathbb{T}$:
\[\pi_2(PU(\mathcal{O}_n))=\mathbb{Z},\quad \pi_1(PU(\mathcal{O}_n))=0\]
(see Corollary \ref{hpu}).

The inner automorphism group $\operatorname{Inn}(\mathcal{O}_n)$ and projective unitary group $PU(\mathcal{O}_n)$ are the same set,
but their topologies are different.
Furthermore, their homotopy types are different.
\begin{thm}\label{M}
    Inclusion $\operatorname{Inn}(\mathcal{O}_n)\subset \operatorname{Aut}(\mathcal{O}_n)$ gives the following isomorphisms:
    \[\pi_i(\operatorname{Inn}(\mathcal{O}_n))\cong \pi_i(\operatorname{Aut}( \mathcal{O}_n)),\quad i\geq 0.\]
\end{thm}
Izumi--Matui's computation \cite{IM} helps us to show a key step of the proof.
\begin{rem}
    Dadarlat--Pennig's argument \cite[Lem. 2.2]{DP} shows $\pi_i(\operatorname{Inn}(\mathcal{O}_\infty))=0=\pi_i(\operatorname{Aut}(\mathcal{O}_\infty))$ (see also Remark \ref{ic}).
\end{rem}

\begin{rem}
We should note that N. Ozawa's surprising result \cite{O} (see also \cite{Dj}) tells us  that the situation is completely different in the case of von Neumann algebra where the topology of (inner) automorphism group is different from and more difficult than the point wise norm topology.
As a corollary of \cite{O},  any von Neumann algebra having separable predual and  tensorial absorption of the hyper finite $II_1$ factor has a contractible inner automorphism group.
The C*-algebraic counter part of this theorem could be contractibility of $\operatorname{Inn}(\mathcal{O}_n\otimes M_{n^\infty})=\operatorname{Inn}(\mathcal{O}_n)$ that never happens for at least point wise norm topology.
\end{rem}
\begin{rem}
    Our result implies that the natural continuous map $U(\mathcal{O}_n)\to\operatorname{Inn}(\mathcal{O}_n)$ is not a fibration due to the topological difference between $PU(\mathcal{O}_n)$ and $\operatorname{Inn}(\mathcal{O}_n)$ resulting from nontrivial central sequences (see Remark \ref{not}).
    Furthermore, the map $\mathbb{Z}/(n-1)\mathbb{Z}=\pi_1(U(\mathcal{O}_n))\to \pi_1(\operatorname{Inn}(\mathcal{O}_n))=\mathbb{Z}/(n-1)\mathbb{Z}$ is zero because it factor through $\pi_1(PU(\mathcal{O}_n))=0$.
\end{rem}

\section*{Aknowledgement}
The author would like to thank Narutaka Ozawa for kindly explaining \cite{O} and telling the author that Dadarlat--Pennig's argument showing the contracibility of $\operatorname{Aut}(\mathcal{O}_2)$ can be directly applicable to $\operatorname{Inn}(\mathcal{O}_2)$.
His comment leads to the question of comparing $PU(\mathcal{O}_2)\sim_hK(\mathbb{Z}, 2)$ and $\operatorname{Inn}(\mathcal{O}_2)\sim_h \text{Point}$.
The author also would like to thank  Ulrich Pennig for the stimulating discussions.

\section{Notations}
In this paper,
every C*-algebra $A$ is assumed to be separable and unital except for some algebras including the algebra $\mathbb{K}$ of compact operators on the infinite dimensional separable Hilbert space and the multiplier algebra $\mathcal{M}(A\otimes \mathbb{K})$ of $A\otimes\mathbb{K}$ (see \cite[Chap. VI]{B}).

For a unital C*-algebra $A$,
we denote by $1_A\in A$ the unit of $A$ and write
\[U(A):=\{u\in A\mid uu^*=u^*u=1_A\}.\]
The unitary group $U(A)$ is a topological group with the norm topology.
An automorphism of $A$ is inner if it is given by
\[\operatorname{Ad}u : A\ni a\mapsto uau^*\in A,\quad u\in U(A).\]
We denote by $\operatorname{Aut}(A)$ (resp. $\operatorname{Inn}(A)$) the group of automorphisms (resp. inner automorphisms) with the point wise topology whose open subbasis are given by the following sets:
\[\{\alpha \in \operatorname{Aut}(A)\mid ||\alpha(a)-\beta(a)||<\epsilon\},\quad \beta\in \operatorname{Aut}(A),\; a\in A,\; \epsilon>0.\]
This topology is metrizable via the metric
\[d(\alpha, \beta):=\sum_{k=1}^\infty\frac{||\alpha(a_n)-\beta(a_n)||}{2^n},\quad \{a_n\}_{n=1}^\infty\subset \{a\in A\mid ||a||\leq 1\}\;\text{dense subset}.\]
The topology of $\operatorname{Inn}(A)$ is the restriction of that of $\operatorname{Aut}(A)$.
We write $\mathbb{T}:=\{z\in\mathbb{C}\mid |z|=1\}$ and often identify it with $\mathbb{T}1_A\subset U(A)$.
If $A$ is simple,
we write $PU(A):=U(A)/\mathbb{T}$ and the topology of $PU(A)$ is  the quotient topology induced by the norm topology of $U(A)$.
Note that the quotient topology is metrizable via the metric
\[d(x, y):=\min\{||u-v||\mid u, v\in U(A),\; x=u\mathbb{T},\; y=v\mathbb{T}\},\quad x, y\in PU(A).\]

We refer to \cite{B} for the notation and basics of KK-theory.
For two C*-algebras $A, B$ and a $*$-homomorphism $\varphi : A\to B$,
we denote by
\[KK(\varphi)\in KK(A, B)\]
the KK-group and Kasparov module given by $\varphi$ (cf. \cite[Ex. 17.1.2]{B}, \cite[Sec. 2]{Dh}).
We denote by $K_i(A)$ the $i$-th K-group.
Note that the Bott periodicity gives us
\[K_{2k}(A)=K_0(S^{2k}A)=K_0(A),\quad K_{2k+1}(A)=K_1(S^{2k}A)=K_1(A),\; K_1(SA)=K_0(A)\]
where $S^iA:=C_0(0, 1)^{\otimes i}\otimes A$ is the $i$-fold suspension of $A$.
The algebra $C_0(0, 1)$ consists of the continuous functions on $[0, 1]$ vanishing at $\{0, 1\}$.
For a compact Hausdorff topological space $X$ (resp.  the sapce with base point $(X, *)$),
we denote by $C(X)$  (resp. $C_0(X, *)$) the algebra of continuous functions (resp. the algebra of continuous functions vanishing at $*$).
Note that the $i$-fold suspension is identified with the functions on the based $i$-dimensional sphere $(S^i, *)$:
\[S^i=C_0(0, 1)^{\otimes i}\cong C_0(S^i, *).\]

For a metrizable topological group $G$ with  a base point ${\rm id}\in G$,
the i-th homotopy group of $G$ is defined by
\[\pi_i(G):=\{\gamma : (S^i, *)\to (G, {\rm id})\mid \gamma (*)={\rm id}\}/\sim_{\text{h}}\]
where 
$\sim_h$ is the base point preserving homotopy equivalence.
So two based maps $\gamma_0, \gamma_1 : (S^i, *)\to (G, {\rm id})$ are equivalent if there exists a continuous map
\[h  : [0, 1]\times S^i\ni (t, x)\mapsto h(t, x)\in G,\quad h(0, x)=\gamma_0(x), \;h(1, x)=\gamma_1(x),\quad h(t, *)={\rm id}\]
(see \cite[Chap. 6]{DK}).
We denote by $[\gamma]\in \pi_i(G)$ the euivalence class of map $\gamma : S^i\to G$.
We write
\[\Omega G:=\{\gamma : S^1\to G\mid \gamma (*)={\rm id}\}=\{\gamma : [0, 1]\to G\mid \gamma(0)=\gamma(1)={\rm id}\},\]
and the following isomorphism
\[\pi_i(\Omega G)\cong \pi_{i+1}(G)\]
holds (see \cite[Chap. 6]{DK}, Appendix).

We denote by
\[[S^i, G]:=\{\gamma : S^i\to G\}/\sim_h\]
the set of homotopy equivalence classes of the continuous maps from $S^i$ to $G$,
where the homotopy may not preserve the base points,
and there is a natural map
\[\pi_i(G)\to [S^i, G]\]
which is a group homomorphism with respect to the group structure of $G$.
\begin{lem}\label{nb}
    If $G$ is path connected,
    the natural group homomorphism 
    \[\pi_i(G)\to [S^i, G]\]
    is bijective.
\end{lem}
\begin{proof}
First, we show surjectivity.
    For any continuous map $\gamma : S^i\to G$,
    one can find a continuous path $\{g_t\}_{t\in [0, 1]}\subset G$ with $g_0=\gamma (*), g_1={\rm id}$.
    Thus, one has \[[\gamma]=[g_1^{-1}\gamma]=[g_0^{-1}\gamma]\in\operatorname{Im}(\pi_i(G)\to [S^i, G]),\]
    and the map $\pi_i(G)\to [S^i, G]$ is surjective.

    Next, we show injectivity.
    Fix a continuous map $f : S^i\to G$ with $f(*)={\rm id}$,
    and assume that $[f]=[l]\in [S^i, G]$ where $l : S^i\ni x\mapsto {\rm id}\in G$ is a constant map.
    Then, there exists a continuous map
    \[h : [0, 1]\times S^i\ni (t, x)\mapsto h(t, x)\in G\]
    satisfying $h(0, x)=l(x), h(1, x)=f(x)$.
    Since $h(0, *)=l(*)={\rm id}=f(*)=h(1, *)$, the following continuous map
    \[[0, 1]\times S^i\ni (t, x)\mapsto h(t, *)^{-1}h(t, x)\in G\]
    gives a base point preserving homotopy equivalence between $f$ and $l$ (i.e., $[f]=[l]=0\in \pi_i(G)$).
    This completes the proof.
\end{proof}
\section{Preliminaries}\label{Pre}
In the sequel,
the C*-algebra $A$ is always assumed to be unital and simple.
Since the map $U(A)\ni u\mapsto \operatorname{Ad}u\in \operatorname{Inn}(A)$ is continuous,
one has a continuous bijection
\[PU(A)\to \operatorname{Inn}(A)\]
which is not homeomorphic in general because of the existence of non-trivial central sequences (see Remark \ref{not}).
\begin{ex}
Let $\mathbb{M}_n(\mathbb{C})$ be the algebra of $n\times n$-matrices.
    Since $U(\mathbb{M}_n(\mathbb{C}))$ is compact,
    the map $PU(\mathbb{M}_n(\mathbb{C}))\to \operatorname{Inn}(\mathbb{M}_n(\mathbb{C}))=\operatorname{Aut}(\mathbb{M}_n(\mathbb{C}))$
    is a homeomorphism.
\end{ex}

A unital C*-algebra $A$ is called Kirchberg algebra if it is separable, nuclear, simple and purely infinite simple (see \cite{R} for more details).
\begin{rem}\label{not}
It is well-known that every Kirchberg algebra tensorially absorbs $\mathcal{O}_\infty$:
\[\varphi : A\cong A\otimes (\mathcal{O}_\infty^{\otimes \mathbb{N}})=A\otimes\left( \bigotimes_{k=1}^\infty\mathcal{O}_\infty\right),\]
and there is a non-trivial central sequence of unitaries $U_n$ satisfying
\[U_n:=\varphi^{-1}(1_A\otimes (1_{\mathcal{O}_\infty}^{\otimes n-1}\otimes u\otimes 1_{\mathcal{O}_\infty}^{\otimes \infty})),\quad u\in U(\mathcal{O}_\infty), \quad C(\mathbb{T})=C^*(\{u\})\subset\mathcal{O}_\infty,\]
\[\lim_{n\to 0}||U_na-aU_n||=0,\quad ||U_n-z1_A||=||u-z||=2,\quad\text{for any}\;\; a\in A,\;\;z\in\mathbb{T}.\]
Since $\operatorname{Ad}U_n\to {\rm id}_A$ in $\operatorname{Inn}(A)$ and $U_n$ does not converge to a scalar,
the continuous bijection $PU(A)\to \operatorname{Inn}(A)$ is not homeomorphism for any unital Kirchberg algebras.
\end{rem}
For a unital Kirchberg algebra $A$,
one can compute the homotopy groups of $PU(A)$ by using the following.
\begin{thm}[{see \cite[Lem. 2.1.7]{P}, \cite[Lem. 2.1]{IM}}]\label{hg}
    For a unital Kirchberg algebra $A$,
    one has
    \[[S^i, U(A)]=K_1(C(S^i)\otimes A), \quad i\geq 0,\]
    \[\pi_i(U(A))=K_1(S^iA),\quad i\geq 0.\]
\end{thm}
\begin{prop}[{see Appendix}]\label{se}
Let $A$ be a unital simple C*-algebra.
    The map $U(A)\to PU(A)$ is a principal $\mathbb{T}$-bundle.
    In particular, this is a Serre fibration and one has the long exact sequence of homotopy groups:
    \[\cdots\to \pi_i(\mathbb{T})\to\pi_i(U(A))\to\pi_i(PU(A))\to\pi_{i-1}(\mathbb{T})\to\pi_{i-1}(U(A))\to\cdots\to \pi_1(PU(A))\to 0\]
\end{prop}

A typical example of Kirchberg algebra is the Cuntz algebra $\mathcal{O}_n,\;\; (n=2, 3, \ldots, \infty)$.
This is the following universal C*-algebra (see \cite{C}):
\[\mathcal{O}_n:=C^*_\text{univ}(\{S_1,\ldots, S_n\mid S^*_iS_j=\delta_{i, j},\; \sum_{i=1}^nS_iS_i^*=1_{\mathcal{O}_n}\}),\]
\[\mathcal{O}_\infty:=C^*_\text{univ}(\{S_1, S_2, \ldots\mid S^*_iS_j=\delta_{i, j}, \; i\in \mathbb{N}\}).\]
The K-theory of the Cuntz algebras are computed in \cite{C1} as follows:
\[K_0(\mathcal{O}_n)=\mathbb{Z}/(n-1)\mathbb{Z},\quad K_1(\mathcal{O}_n)=0,\]
\[K_0(\mathcal{O}_\infty)=\mathbb{Z},\quad K_1(\mathcal{O}_\infty)=0.\]
Combining Theorem \ref{hg}, Proposition \ref{se} and the above K-groups,
one has the following.
\begin{cor}\label{hpu}
For $n<\infty$, one has
    \[\pi_0(PU(\mathcal{O}_n))=\pi_1(PU(\mathcal{O}_n))=0,\quad \pi_2(PU(\mathcal{O}_n))=\mathbb{Z},\quad \pi_{k\geq 3}(PU(\mathcal{O}_n))=K_{k+1}(\mathcal{O}_n).\]
\end{cor}
In the rest of this section,
we recall 
Izumi--Matui's computation of $\pi_i(\Omega\operatorname{Aut}(A\otimes \mathbb{K}))$.
In \cite{IM},
they show that every element of $\pi_i(\Omega\operatorname{Aut}(A\otimes \mathbb{K})), \quad i\geq 0$ is given by the following data:
\begin{equation} u : [1/2, 1)\times S^i\ni (t, x)\mapsto u(t, x)\in U(A) \label{central}
\end{equation}
\begin{equation*}
    u(t, *)=1_A,\quad \sup_{x\in S^i}||u(t, x)a-au(t, x)||\to 0,\;\; (t\to 1),
\end{equation*}
\begin{equation}v : [0, 1/2]\times S^i\to U(\mathcal{M}(A\otimes \mathbb{K})),\label{pa2}
\end{equation}
\begin{equation*}v(0, x)=v(t, *)=1_{A}\otimes 1_{\mathcal{M}(\mathbb{K})},\quad v(1/2, x)=u(1/2, x)\otimes 1_{\mathcal{M}(\mathbb{K})},
\end{equation*}
where all maps are norm continuous.
From the above unitaries,
one has the following automorphisms
\[\operatorname{Ad}U(t, x)\in \operatorname{Aut}(A\otimes\mathbb{K}),\]
\[U(t, x):=v(t, x),\;\; (0\leq t\leq 1/2),\quad U(t, x):=u(t, x)\otimes 1_{\mathcal{M}(\mathbb{K})},\;\; (1/2\leq t < 1),\]
and this gives a map
\[\rho^u : S^i\ni x\mapsto (t\mapsto \operatorname{Ad} U(t, x))\in\Omega\operatorname{Aut}(A\otimes\mathbb{K}),\]
satisfying
\[\rho^u_x : t\mapsto \operatorname{Ad}U(t, x),\;\; 0\leq t<1,\quad\rho^u_x : 1\mapsto {\rm id}_{A\otimes\mathbb{K}},\]
which is well-defined because
\[\operatorname{Ad} U(0, x)=\operatorname{Ad}U(t, *)={\rm id}_{A\otimes\mathbb{K}},\]
\[\operatorname{Ad}U(t, x)(a\otimes k)=(u(t, x)au(t, x)^*)\otimes k\to a\otimes k,\;\; (t\to 1), \;\; a\in A,\; k\in\mathbb{K}.\]

\begin{thm}[{\cite[Thm. 2.6, Lem. 2.3.]{IM}}]\label{loop}
Let $A$ be a unital Kirchberg algebra.
   Every element of $\pi_i(\Omega\operatorname{Aut}(A\otimes \mathbb{K}))\; (i\geq 0)$ is given by $[\rho^u]$ and the class $[\rho^u]$ only depends on the homotopy class of $u$ in eq. (\ref{central}) and does not depend on the choice of $v$ in eq. (\ref{pa2}).
\end{thm}
\section{Proof of Theorem \ref{M}}
We need several lemmas.
\begin{lem}\label{inj}
    There exists a continuous map
    \[[0, 1)\ni t\mapsto J_t\in \operatorname{End}(\mathcal{O}_\infty)\]
    satisfying
    \[J_0={\rm id}_{\mathcal{O}_\infty},\quad||J_t(b)a-aJ_t(b)||\to 0,\quad (t\to 1),\quad a, b\in\mathcal{O}_\infty,\]
   where $\operatorname{End}(\mathcal{O}_\infty)$ is the set of unital $*$-endomorphisms equipped with the point wise norm topology. 
\end{lem}
The above lemma is well-known and we write proof for the convenience of readres.
\begin{proof}
    By identifying $\varphi : \mathcal{O}_\infty\cong \mathcal{O}_\infty^{\otimes \mathbb{N}}$,
    the embeddings
    \[j_n : \mathcal{O}_\infty\ni x\mapsto 1^{\otimes n-1}_{\mathcal{O}_\infty}\otimes x\otimes 1_{\mathcal{O}_\infty}^{\otimes\infty}\in\mathcal{O}_\infty^{\otimes\mathbb{N}},\quad n\in\mathbb{N},\]
    gives a unital endomorphisms $\varphi^{-1}\circ j_n\in \operatorname{End}(\mathcal{O}_\infty)$.
    Since \[KK(1_{\mathcal{O}_\infty}\otimes {\rm id}_{\mathcal{O}_\infty})=KK({\rm id}_{\mathcal{O}_\infty}\otimes 1_{\mathcal{O}_\infty})=1\in KK(\mathcal{O}_\infty, \mathcal{O}_\infty^{\otimes 2})=\mathbb{Z},\]
    Kirchberg--Phillips' theorem (cf. \cite[Sec. 2]{Dh}) and $\pi_0(U(\mathcal{O}_\infty))=K_1(\mathcal{O}_\infty)=0$ give a continuous path
\[\alpha_t : \mathcal{O}_\infty\to\mathcal{O}_\infty^{\otimes 2},\quad t\in [0, 1],\;\;\alpha_0={\rm id}\otimes 1,\;\;\alpha_1=1\otimes{\rm id}.\]
    Thus, one obtains the continuous paths
    \[[n, n+1]\ni t\mapsto \varphi^{-1}\circ j_t\in \operatorname{End}(\mathcal{O}_\infty),\quad n\in \mathbb{N},\]
     satisfying
     \[j_t(\mathcal{O}_\infty)\subset 1_{\mathcal{O}_\infty}^{\otimes n-1}\otimes \mathcal{O}_\infty^{\otimes 2}\otimes 1_{\mathcal{O}_\infty}^{\otimes\infty}\subset \mathcal{O}_\infty^{\otimes\mathbb{N}}.\]
     Kirchberg--Phillips's theorem also gives a path
    \[[0, 1]\ni t\mapsto \psi_t\in \operatorname{End}(\mathcal{O}_\infty),\quad \psi_0={\rm id}_{\mathcal{O}_\infty},\;\; \psi_1=\varphi^{-1}\circ j_1.\]
    Now one obtains the desired map $J :[0, 1)\cong [0, +\infty)\to\operatorname{End}(\mathcal{O}_\infty)$ by concatenating the above paths.
\end{proof}
\begin{lem}\label{fh}
    The map $\pi_i(\operatorname{Inn}(\mathcal{O}_n))\to\pi_i(\operatorname{Aut}(\mathcal{O}_n)),\;\; (i\geq 0)$ is injective.
\end{lem}
\begin{proof}
We identify $\mathcal{O}_n$ with $\mathcal{O}_n\otimes \mathcal{O}_\infty$
and  show that $\pi_i(\operatorname{Inn}(\mathcal{O}_n\otimes\mathcal{O}_\infty))\to\pi_i(\operatorname{Aut}(\mathcal{O}_n\otimes\mathcal{O}_\infty))$ is injective.

    By Theorem \ref{hg} (i.e., $\pi_0(U(\mathcal{O}_n))=K_1(\mathcal{O}_n)=0$) and \cite[Thm. 7.4.]{Dc},
    two groups $\operatorname{Inn}(\mathcal{O}_n), \operatorname{Aut}(\mathcal{O}_n)$ are path connected.
    By Lemma \ref{nb},
    it is enough to show the injectivity of
    \[[S^i, \operatorname{Inn}(\mathcal{O}_n\otimes\mathcal{O}_\infty)]\to[S^i, \operatorname{Aut}(\mathcal{O}_n\otimes\mathcal{O}_\infty)].\]
    
    Fix $\gamma : S^i\ni x\mapsto \gamma_x\in \operatorname{Inn}(\mathcal{O}_n\otimes\mathcal{O}_\infty)$ and write $l : S^i\ni x\mapsto{\rm id}\in\operatorname{Aut}(\mathcal{O}_n\otimes\mathcal{O}_\infty)$.
    Note that Lemma \ref{nb} implies that the group $[S^i, G](\cong \pi_i(G))$ is commutative and the unit is the equivalence class of $l : S^i\ni x\mapsto {\rm id}\in G$.
    
As in \cite[Sec. 4]{Dh},
we identify $\gamma, l$ with the unital $*$-homomorphisms
\[\gamma : \mathcal{O}_n\otimes\mathcal{O}_\infty\ni d\mapsto (x\mapsto \gamma_x(d))\in C(S^i)\otimes\mathcal{O}_n\otimes\mathcal{O}_\infty,\]
  \[l : \mathcal{O}_n\otimes\mathcal{O}_\infty\ni d\mapsto 1_{C(S^i)}\otimes d\in C(S^i)\otimes\mathcal{O}_n\otimes\mathcal{O}_\infty.\]  
    Assume that $[\gamma]=[l](=0)\in [S^i, \operatorname{Aut}(\mathcal{O}_n\otimes\mathcal{O}_\infty)]$.
    Then, one has \[KK(\gamma)=KK(l)\in KK(\mathcal{O}_n\otimes\mathcal{O}_\infty, C(S^i)\otimes\mathcal{O}_n\otimes\mathcal{O}_\infty),\]
    \[\gamma (1)=l(1)=1_{C(S^i)}\otimes 1,\]
    and \cite[Thm. 2.9, Lem. 3.1]{Dh} gives a continuous path
    \[ [0, 1)\ni t\mapsto u_t\in U(C(S^i)\otimes\mathcal{O}_n\otimes\mathcal{O}_\infty)\]
    satisfying
    \[||\operatorname{Ad}u_t\circ\gamma (d)-l(d)||\to 0,\quad (t\to 1), \;\; d\in\mathcal{O}_n\otimes\mathcal{O}_\infty.\]

    Thus, we have
    \[[\operatorname{Ad}u_0]+[\gamma]=[\operatorname{Ad}u_0\circ\gamma]=[l]=0\in [S^i, \operatorname{Inn}(\mathcal{O}_n\otimes\mathcal{O}_\infty)],\]
    where the map $\operatorname{Ad}u_0$ is given by
    \[\operatorname{Ad}u_0 : S^i\ni x\mapsto \operatorname{Ad}u_0(x)\in \operatorname{Inn}(\mathcal{O}_n\otimes\mathcal{O}_\infty),\]
    \[u_0 : S^i\ni x\mapsto u_0(x)\in U(\mathcal{O}_n\otimes\mathcal{O}_\infty),\quad u_0\in U(C(S^i)\otimes\mathcal{O}_n\otimes\mathcal{O}_\infty).\]

    Since the map $\theta : \mathcal{O}_\infty\ni b\mapsto 1\otimes b\in \mathcal{O}_n\otimes\mathcal{O}_\infty$ induces the surjection
    \[K_1(C(S^i)\otimes\mathcal{O}_\infty)\to K_1(C(S^i)\otimes\mathcal{O}_n\otimes\mathcal{O}_\infty),\quad (i\geq 0),\]
    there exists an unitary $v\in C(S^i)\otimes\mathcal{O}_\infty$ such that two unitaries \[({\rm id}_{C(S^i)}\otimes \theta)(v),\; u_0\in U(C(S^i)\otimes\mathcal{O}_n\otimes\mathcal{O}_\infty)\] are homotopic by Theorem \ref{hg}.
    Thus, one has
    \[[\operatorname{Ad}u_0]=[\operatorname{Ad}({\rm id}_{C(S^i)}\otimes\theta)( v)]\in [S^i, \operatorname{Inn}(\mathcal{O}_n\otimes \mathcal{O}_\infty)].\]
    Using the unital endomorphism $J_t$ in Lemma \ref{inj},
    consider the path
    \[V_t:=({\rm id}_{C(S^i)}\otimes(\theta\circ J_t))(v)\in U(C(S^i)\otimes \mathcal{O}_n\otimes\mathcal{O}_\infty),\; t\in [0, 1),\; V_0=({\rm id}_{C(S^i)}\otimes\theta)(v).\]
    By Lemma \ref{inj},
    one has 
    \begin{align*}
    &||({\rm id}_{C(S^i)}\otimes(\theta\circ J_t))(f\otimes b)(h\otimes d\otimes a)-(h\otimes d\otimes a)({\rm id}_{C(S^i)}\otimes(\theta\circ J_t))(f\otimes b)||\\
    =&||fh\otimes d\otimes (J_t(b)a-aJ_t(b))||\\
    \to&0, \;\; (t\to 1)
    \end{align*}
    for any $f, h\in C(S^i),\; d\in\mathcal{O}_n,\; a, b\in\mathcal{O}_\infty$,
    and this implies
\begin{align*}
    ||\operatorname{Ad}V_t(h\otimes d\otimes a)-h\otimes d\otimes a||\to 0,\;\; (t\to 1).
\end{align*}
    
So, two maps
    $\operatorname{Ad} ({\rm id}_{C(S^i)}\otimes\theta)(v)=\operatorname{Ad}V_0$ and $l$ are connected by the path 
    \[\operatorname{Ad}V_t : \mathcal{O}_n\otimes\mathcal{O}_\infty\ni d\mapsto V_t(1_{C(S^i)}\otimes d)V_t^*\in C(S^i)\otimes\mathcal{O}_n\otimes\mathcal{O}_\infty, \;t\in [0, 1).\]
    Finally, we have \[[\operatorname{Ad}u_0]=[\operatorname{Ad}({\rm id}_{C(S^i)}\otimes\theta)(v)]=[l]=0\in [S^i, \operatorname{Inn}(\mathcal{O}_n\otimes\mathcal{O}_\infty)]\]
    and $[\gamma]=0\in [S^i, \operatorname{Inn}(\mathcal{O}_n\otimes\mathcal{O}_\infty)]$.
    This completes the proof.
\end{proof}
\begin{rem}\label{ic}
    The above argument using $J_t$ and \cite[Thm. 2.3]{DP} implies the weak contractibility of $\operatorname{Inn}(\mathcal{O}_\infty)$.
\end{rem}

By Lemma \ref{fh} and eq. (\ref{dada}), it is enough to show the surjectivity of the map
\[\pi_{2k+1}(\operatorname{Inn}(\mathcal{O}_n))\to \pi_{2k+1}( \operatorname{Aut}(\mathcal{O}_n))=\mathbb{Z}/(n-1)\mathbb{Z},\quad k\geq 0\]
for  completing the proof of Theorem \ref{M}.
\begin{proof}[{Proof of Theorem \ref{M}}]
Let 
\[C_\nu \mathcal{O}_n:=\{f\in C[0, 1]\otimes\mathcal{O}_n\mid f(0)\in \mathbb{C}1_{\mathcal{O}_n},\;\; f(1)=0\}\]
be the mapping cone of the map $\nu : \mathbb{C}1_{\mathcal{O}_n}\subset\mathcal{O}_n$ with the following exact sequence (see \cite[Sec. 19.4]{B}):
\[S\mathcal{O}_n\xrightarrow{i}C_\nu\mathcal{O}_n\xrightarrow{ev_0}\mathbb{C}.\]
Then, \cite[Thm. 19. 5. 7]{B} yields the following exact sequence:
  \[\xymatrix{
    &\mathbb{Z}/(n-1)\mathbb{Z}\ar@{=}[d]&\mathbb{Z}/(n-1)\mathbb{Z}\ar@{=}[d]\\
    KK(\mathbb{C}, SS^{2k+1}\mathcal{O}_n)\ar[r]^{ev_0^*}&KK(C_\nu\mathcal{O}_n, SS^{2k+1}\mathcal{O}_n)\ar[r]^{i^*}&KK(S\mathcal{O}_n, SS^{2k+1}\mathcal{O}_n)\ar[d]\\
    KK^1(S\mathcal{O}_n, SS^{2k+1}\mathcal{O}_n)\ar[u]&KK^1(C_\nu\mathcal{O}_n, SS^{2k+1}\mathcal{O}_n)\ar[l]^{i^*}&KK^1(\mathbb{C}, SS^{2k+1}\mathcal{O}_n),\ar[l]^{ev_0^*}\\
    &&0\ar@{=}[u]
    }\]
    and the map $i^* : KK(C_\nu\mathcal{O}_n, SS^{2k+1}\mathcal{O}_n)\to KK(S\mathcal{O}_n, SS^{2k+1}\mathcal{O}_n)$ is an isomorphism.
    
    Applying \cite[Thm. 5.9]{Dh},
    one has the commutative diagram
    \[\xymatrix{
    \pi_{2k+1}(\operatorname{Aut}(\mathcal{O}_n))\ar[d]^{\cong}\ar[r]^{-\otimes{\rm id}_\mathbb{K}}& \pi_{2k+1}(\operatorname{Aut}(\mathcal{O}_n\otimes\mathbb{K}))\ar[d]^{\cong}\\
    KK(C_\nu \mathcal{O}_n, SS^{2k+1}\mathcal{O}_n)\ar[r]^{i^*}&KK(S\mathcal{O}_n, SS^{2k+1}\mathcal{O}_n) 
    }\]
    where the top horizontal map is induced by
    \[\operatorname{Aut}(\mathcal{O}_n)\ni\alpha\mapsto \alpha\otimes{\rm id}_{\mathbb{K}}\in\operatorname{Aut}(\mathcal{O}_n\otimes\mathbb{K}).\]
    

    
    Now we obtain an isomorphism
    \[-\otimes{\rm id}_{\mathbb{K}} : \pi_{2k+1}(\operatorname{Aut}(\mathcal{O}_n))\cong \pi_{2k+1}(\operatorname{Aut}(\mathcal{O}_n\otimes\mathbb{K}))\cong\pi_{2k}(\Omega\operatorname{Aut}(\mathcal{O}_n\otimes\mathbb{K})),\]
and it is enough to show that the horizontal maps in the commutative diagram below are surjective:
    \[\xymatrix{
    \pi_{2k+1}(\operatorname{Inn}(\mathcal{O}_n))\ar[r]^{-\otimes{\rm id}_{\mathbb{K}}}&\pi_{2k+1}(\operatorname{Aut}(\mathcal{O}_n\otimes\mathbb{K}))\\
    \pi_{2k}(\Omega\operatorname{Inn}(\mathcal{O}_n))\ar[u]^{\cong}\ar[r]^{-\otimes{\rm id}_{\mathbb{K}}}&\pi_{2k}(\Omega\operatorname{Aut}(\mathcal{O}_n\otimes\mathbb{K})).\ar[u]^{\cong}
    }\]
    
    By Theorem \ref{loop},
    every element in $\pi_{2k}(\Omega\operatorname{Aut}(\mathcal{O}_n\otimes\mathbb{K}))$ is given by $[\rho^u]$ with the unitaries $u(t, x)\in U(\mathcal{O}_n),\; v(t, x)\in U(\mathcal{M}(\mathcal{O}_n\otimes\mathbb{K}))$ in eq. (\ref{central}), (\ref{pa2}).
    Since $[u(1/2, \cdot)]\in \pi_{2k}(U(\mathcal{O}_n))=0$ by Theorem \ref{hg},
    one can choose $v(t, x)$ to be
    \[v(t, x):=\tilde{v}(t, x)\otimes 1_{\mathcal{M}(\mathbb{K})}\]
    where $\tilde{v}\in U(C[0, 1/2]\otimes C(S^{2k})\otimes\mathcal{O}_n)$ is a path of unitaries satisfying
    \[\tilde{v}(0, x)= \tilde{v}(t, *)=1_{\mathcal{O}_n},\quad\tilde{v}(1/2, x)=u(1/2, x)\in U(\mathcal{O}_n).\]
    Now the unitaries $u(t, x), \tilde{v}(t, x)\in U(\mathcal{O}_n)$ gives an element 
    \[S^i\ni x\mapsto (t\mapsto \operatorname{Ad}w(t, x))\in\Omega\operatorname{Inn}(\mathcal{O}_n)\]
    of $\pi_{2k}(\Omega\operatorname{Inn}(\mathcal{O}_n))$ defined by
    \[w(t, x):=\tilde{v}(t, x),\;\; t\in [0, 1/2],\quad w(t, x):=u(t, x),\;\; t\in [1/2, 1).\]
    Since this element is sent to $[\rho^u]$ by the map
    \[-\otimes{\rm id}_\mathbb{K} : \pi_{2k}(\Omega\operatorname{Inn}(\mathcal{O}_n))\to \pi_{2k}(\Omega\operatorname{Aut}(\mathcal{O}_n)),\]
    we have proven the surjectivity.
\end{proof}

\section*{Appendix}
We briefly recall some homotopy theory used in this paper.
\subsection{The homotopy groups}
A topological space $X$ is called compactly generated if $X$ is Hausdorff and a subset $A\subset X$ is closed if and only if $A\cap C$ is closed for any compact subset $C\subset X$.
The metric space and CW complex, for example $S^i$, are compactly generated.
In the rest of this section,
we consider a metric space $(X, d_X)$ with a base point $x_0\in X$.

We write
\[S^1\wedge S^i:=\frac{S^1\times S^i}{(S^1\times \{*\}\; \cup \{*\}\times S^i)},\]
\[(t, x)\in S^1\wedge S^i,\quad t\in S^1:=\frac{[0, 1]}{\{0, 1\}},\quad x\in S^i.\]
Note that one has $S^1\wedge S^i\cong S^{i+1}$.
Let denote by
\[\operatorname{Map}_0(S^i, X):=\{\gamma : S^i\to X\mid \gamma (*)=x_0\}\]
the set of continuous base point preserving maps,
and one has to introduce a topology in $\operatorname{Map}_0(S^i, X)$.
\begin{lem}\label{topo}
For a metric space $X$,
  the compact--open topology of $\operatorname{Map}_0(S^i, X)$ coincides with a metric topology defined by
  \[d(\gamma_1, \gamma_2):=\max\{d_X(\gamma_1(x), \gamma_2(x))\mid x\in S^i\},\quad \gamma_1, \gamma_2\in\operatorname{Map}_0(S^i, X).\]
\end{lem}
By the above lemma,
the space $\operatorname{Map}_0(S^i, X)$ with the compact--open topology is metrizable and compactly generated,
and one can define the homotopy group by
\[\pi_i(X):=\operatorname{Map}_0(S^i, X)/\sim_h\]
(see \cite[Sec. 6.1.2]{DK} for the general definition of topology in $\operatorname{Map}_0(-, -)$).

We write $\Omega X:=\operatorname{Map}_0(S^1, X)$ and this is a metric space with a metric as in Lemma \ref{topo}.
Furthermore, the following sets
\[\operatorname{Map}_0(S^1\wedge S^i, X),\quad \operatorname{Map}_0(S^i, \Omega X)\]
are metrizable in a similar manner, and the following map preserves the metrics (i.e., homeomorphism)
\[\operatorname{Map}_0(S^1\wedge S^i, X)\ni f(t, x)\mapsto (x\mapsto (t\mapsto f(t, x)))\in\operatorname{Map}_0(S^i, \Omega X).\]

Thus, one has
\[\pi_{i+1}(X)\cong \pi_i(\Omega X).\]

\subsection{Proof of Proposition \ref{se}}

For $u\in U(A)$,
we denote by \[\sigma (u):=\{z\in\mathbb{C}\mid z1_A-u\in A\;\text{is not invertible}\}\]
the spectrum of $u$.
We write $B_r(u):=\{w\in U(A)\mid ||w-u||<r\}$.
   We denote by
   \[q : U(A)\to PU(A)\]
   the quotient map.
   Note that $PU(A)$ is a topological group with respect to the quotient topology and $q$ is an open map.
   Let $\epsilon>0$ be a small positive real number such that the spectrum of arbitrary $w\in B_{2\epsilon}(1_A)$ is contained in
   \[\{e^{it}\in\mathbb{T}\mid t\in [-\pi/8, \pi/8]\}.\]

   We have an open set
   \[O:=q(B_\epsilon(1_A))\subset PU(A)\]
   and we will construct a continuous section
   \[s : O\to U(A).\]
For any $x\in O$,
there is a lift $v\in B_\epsilon(1_A)$ satisfying $x=q(v)$.
Since the spectrum $\sigma(v)\subset \{e^{it}\mid t\in [-\pi/8,\pi/8]\}$ is compact,
one has 
\[t_v:=\max\{t\in [-\pi/8, \pi/8]\mid e^{it}\in\sigma (v)\}\]
and we define $s$ by
\[s : O\ni x\mapsto e^{-it_v}v\in U(A),\quad v\in B_\epsilon(1_A),\quad q(v)=x.\]
   \begin{lem}\label{wel}
       The map $s$ is well-defined.
   \end{lem}
   \begin{proof}
       Let $v_1, v_2\in B_\epsilon(1_A)$ be two lifts of $x$ (i.e., $q(v_1)=q(v_2)=x$).
       Since $v_1v_2^*\in B_{2\epsilon}(1_A)\cap \mathbb{T}$,
       we may assume that there exists $t_0\in [0,\pi/8]$ satisfying
       \[v_1=e^{it_0}v_2.\]
       Thus, one has \[\sigma (v_1)=e^{it_0}\sigma(v_2)\subset\{e^{it}\mid t\in [-\pi/4, \pi/4]\},\]
       and this implies $t_{v_1}=t_0+t_{v_2}$.
       So one has
       \begin{align*}
           e^{-it_{v_1}}v_1=e^{-it_{v_2}}e^{-it_0}e^{it_0}v_2=e^{-it_{v_2}}v_2
       \end{align*}
       and the section $s$ is well-defined.
   \end{proof}
   \begin{lem}\label{cont}
       The section $s : O\to U(A)$ is continuous.
   \end{lem}

   \begin{proof}
       Fix an open set $B_\delta(u),\;\delta>0,\; u\in U(A)$ and $x\in s^{-1}(B_\delta (u))\subset O$. 
There exists $v\in B_\epsilon(1_A)$ satisfying
\[s(x)=e^{-it_v}v,\quad ||s(x)-u||<\delta.\]
If $w\in B_\epsilon(v)$,
one has \[w\in B_{2\epsilon}(1_A),\;\; \sigma (w)\subset \{e^{it}\mid t\in[-\pi/8, \pi/8]\}.\]
For the function \[\text{Arg} : \{e^{it}\mid t\in [-\pi/8, \pi/8]\} \ni e^{it}\mapsto t\in (-\pi, \pi],\]
 functional calculus gives
\[t_w=\sup\{ t\in \sigma (\text{Arg}(w))\},\quad w\in B_\epsilon (v).\]
Note that 
\[B_\epsilon(v)\ni w\mapsto \text{Arg}(w)\in A\]
is continuous because evrery spectrum $\sigma (w)$ is contained in the domain $\{e^{it}\mid t\in [-\pi/8, \pi/8]\}$ and $\text{Arg}$ is uniformly approximated on the domain via the trigonometric polynomials.

For a positive real number $\delta_1:=\frac{\delta-||s(x)-u||}{4}$,
there exists $\epsilon>\delta_2>0$ such that every $w\in B_{\delta_2}(v)$ satisfies
\[|t_v-t_w|<\delta_1\]
due to the continuity of functional calculus.
Let $\delta_3:=\min\{\delta_1, \delta_2, \frac{\epsilon-||v-1_A||}{4}\}>0$.
We show that 
\[x\in q(B_{\delta_3}(v))\subset s^{-1}(B_\delta(u)).\]
For any $w\in B_{\delta_3}(v)$,
one has
\begin{align*}
    ||w-1_A||=&||w-v+v-1_A||\\
    <&\delta_3+||v-1_A||\\
    \leq&\frac{\epsilon-||v-1_A||}{4}+||v-1_A||\\
    <&\epsilon
\end{align*}
and $x=q(v)\in q(B_{\delta_3}(v))\subset O$.
For any $q(w)\in q(B_{\delta_3}(v))$,
one has
\begin{align*}
    ||s(q(w))-u||=&||e^{-it_w}w-u||\\
    =&||e^{-it_w}w-e^{-it_v}v+s(x)-u||\\
    \leq&||e^{-it_w}w-e^{-it_v}v||+||s(x)-u||\\
    \leq&||e^{-it_w}w-e^{-it_v}w||+||e^{-it_v}w-e^{-it_v}v||+||s(x)-u||\\
    \leq&|t_w-t_v|+||w-v||+||s(x)-u||\\
    \leq&\delta_1+\delta_3+||s(x)-u||\\
    \leq&\frac{\delta-||s(x)-u||}{2}+||s(x)-u||\\
    <&\delta
\end{align*}
and this implies 
\[q(B_{\delta_3}(v))\subset s^{-1}(B_\delta(u)).\]
Applying above argument for arbitrary $x\in s^{-1}(B_\delta(u))$ shows $s^{-1}(B_\delta(u))\subset O$ is open,
and we have shown that $s : O\to U(A)$ is continuous.
   \end{proof}

\begin{proof}[{Proof of Proposition \ref{se}}]
Since $PU(A)$ is a topological group,
any element $y=q(w)\in PU(A)$ has an open neighborhood $Oy:=\{xy\in PU(A)\mid x\in O\}$,
and Lemma \ref{wel} and Lemma \ref{cont} give a continuous section $S_y : Oy\ni xy\mapsto s(x)w\in U(A)$.
Then, one has a local trivialization
\[q^{-1}(Oy)\ni u\mapsto (q(u), S_y(q(u))^*u)\in Oy\times\mathbb{T},\]
\[Oy\times \mathbb{T}\ni (x, z)\mapsto S_y(x)z\in q^{-1}(Oy)\]
which are conpatible with $\mathbb{T}$-action.
Thus, the map $q : U(A)\to PU(A)$ is a locally trivial principal $\mathbb{T}$-bundle.

Let $PU(A)_0, U(A)_0$ be the path connected components of the identities.
By definition, one has \[\pi_i(PU(A))=\pi_i(PU(A)_0),\quad \pi_i(U(A))=\pi_i(U(A)_0),\quad i\geq 1,\]
and \cite[Thm. 4.41, Prop. 4.48]{Hat} yield a long exact sequence
\[\cdots\to\pi_{i+1}(PU(A))\to\pi_i(\mathbb{T})\to \pi_i(U(A))\to \pi_i(PU(A))\to\cdots \to\pi_1(PU(A))\to \pi_0(\mathbb{T})\to\pi_0(U(A)_0)\to 0. \]
Since $\pi_0(\mathbb{T})=\pi_0(U(A)_0)=0$,
this completes the proof.
\end{proof}
We note that the fiber bundle $U(A)\to PU(A)$ has the homotopy lifting property for $[0, 1]$ and this implies
\[U(A)_0=q^{-1}(PU(A)_0),\quad \pi_0(U(A))=\pi_0(PU(A)).\]

\end{document}